\documentclass[11pt]{article}

\usepackage{amsmath,amssymb,amsthm}
\usepackage{graphicx}
\usepackage{subfigure}
\usepackage{enumerate}
\usepackage{fullpage}
\usepackage{url}
\usepackage{float}
\usepackage{xspace}
\usepackage{hyperref}
\usepackage{color}
\usepackage{thmtools, thm-restate}

\newtheorem{theorem}{Theorem}

\newtheorem{corollary}{Corollary}
\newtheorem{lemma}{Lemma}[section]
\newtheorem{proposition}{Proposition}
\newtheorem{question}{Question}
\newtheorem{remark}{Remark}

\title{Contact graphs of ball packings}

\author{Alexey Glazyrin \thanks{School of Mathematical \& Statistical Sciences,
         The University of Texas Rio Grande Valley, USA}}

\date{}%\today}

\begin{document}

\maketitle

\begin{abstract}
A contact graph of a packing of closed balls is a graph with balls as vertices and pairs of tangent balls as edges. We prove that the average degree of the contact graph of a packing of balls (with possibly different radii) in $\mathbb{R}^3$ is not greater than $13.92$. We also find new upper bounds for the average degree of contact graphs in $\mathbb{R}^4$ and $\mathbb{R}^5$.
\end{abstract}


\section{Introduction}\label{sect:intro}

A packing of closed balls in $\mathbb{R}^d$ is a finite set of balls with non-intersecting interiors. Each packing naturally entails a contact graph where graph vertices are the balls of the packing and two vertices are connected by an edge if and only if the corresponding balls are tangent.

The problem of characterizing contact graphs of planar disk packings is completely solved by the Koebe--Andreev--Thurston Theorem (\cite{koe36,and70a,and70b,thu88}).

\begin{theorem}[Koebe--Andreev--Thurston]
For every simple planar graph $G$ there is a set of non-intersecting closed disks on the plane whose contact graph is $G$.
\end{theorem}

The natural question is to get a similar characterization of contact graphs in higher dimensions.

\begin{question}
What graphs may be realized as contact graphs of closed balls in $\mathbb{R}^d$?
\end{question}

This question does not impose any restrictions on closed balls. One of such reasonable restrictions is to require balls in a packing to be congruent. An observation by Kirkpatrick and Rote (see \cite{hli01} for more details) establishes that a graph $G$ is a contact graph of a unit ball packing in $\mathbb{R}^d$ if and only if the join $G\oplus K_2$ of the graph $G$ with an edge $K_2$ is a contact graph of a general ball packing in $\mathbb{R}^{d+1}$. This observation combined with results for packings of unit balls imply that recognizing a contact graph of ball packings is NP-hard in dimensions 3, 4, 5, 9, 25 (\cite{hli97,bre96,hli01}, see also the survey \cite{bez18}).

Since, as we can see, the general question is quite complicated, typically some characteristics of contact graphs are analyzed. For each contact graph $G$ of a closed ball packing in $\mathbb{R}^d$ denote its average degree by $k(G)$. Define $k_d=\sup k(G)$ taken over all contact graphs. A simple way to bound $k_d$ is by using kissing numbers. By a kissing number $\tau_d$ we mean the maximum number of non-overlapping closed unit balls tangent to a given unit ball in $\mathbb{R}^d$. As mentioned in \cite{kup94}, it easy to show that $k_d\leq 2\tau_d,$ since each ball cannot have more than $\tau_d$ larger or equal balls tangent to it. The state-of-the-art bounds for $\tau_d$ imply the following bounds for $k_d$:

$$k_3\leq 24 \cite{sch53}; k_4\leq 48 \cite{mus08}; k_5\leq 88 \cite{mit10}; k_6\leq 136 \cite{bac08}; k_7\leq 268 \cite{mit10}; k_8 \leq 480 \cite{odl79,lev79}$$
and the asymptotic (Kabatyanskii-Levenshtein) bound $$k_d\leq 2^{0.401d(1+o(1))} \cite{kab78}.$$

Using the area argument Kuperberg and Schramm proved in \cite{kup94} a non-trivial upper bound for $k_3$. They also connected $600$-cells in a chain and produced a ball packing with the average degree strictly greater than $12$.

\begin{theorem}[Kuperberg--Schramm \cite{kup94}]
$$12.56 \approx 666/53 \leq k_3 < 8 + 4\sqrt{3} \approx 14.93.$$
\end{theorem}

In $\cite{epp03}$, the lower bound for $k_3$ was improved by a more intricate construction based on $600$-cells.

\begin{theorem}[Eppstein--Kuperberg--Ziegler \cite{epp03}]
$$12.61 \approx 7656/607 \leq k_3.$$
\end{theorem}

In the general case, $k_d$ can be bounded below by a lattice kissing number $\tau_d^*$, the maximal number of balls tangent to one ball in a lattice packing of unit balls. %Unfortunately, no general lower bounds for $\tau_d^*$ with exponential growth are known.
For $d=2^n$, it was shown that $\tau_d^*=2^{\Omega(\log^2 d)}$ \cite{lee64}. As for non-lattice packings, it was proven in \cite{alo97} that there is a finite unit ball packing in dimension $d=4^n$ such that each ball touches more than $2^{\sqrt{d}}$ others. This result implies $k_{4^n}>2^{2^n}$. Recently, Vl\u adu\c t \cite{vla19} proved the first exponential lower bound for lattice kissing numbers $\tau_d^* \geq 2^{0.0219d(1-o(1))}$. Various studies were also dealing with certain characteristics of contact graphs of ball packings such as chromatic numbers (\cite{mae07,che17}), graph separators (\cite{mil97}), or with non-realizability of concrete graphs (\cite{ben13}).

The main results of this paper are the new upper bounds $k_3<13.92$, $k_4<34.69$, $k_5<77.76$. The results in dimensions 4 and 5 are obtained by generalizing the method of Kuperberg and Schramm to higher dimensions. The improvement of the bound in $\mathbb{R}^3$ required a thorough analysis of packings of spherical caps via the upper bounds from \cite{flo01, flo07}.

The paper is structured as follows. In Section \ref{sect:kup-sch} we explain the approach of Kuperberg and Schramm with their upper bound for $k_3$. In Section \ref{sect:high} we show how their approach works in higher dimensions. Section \ref{sect:3d} is devoted to the proof of the new upper bound for $k_3$. Finally, in Section \ref{sect:discuss} we raise some questions and discuss possible future advancements in this area.

\section{Kuperberg-Schramm approach}\label{sect:kup-sch}

In this section we will explain the approach of Kuperberg and Schramm which allowed them to prove the upper bound of $8+4\sqrt{3}$ for $k_3$. Throughout the section we will use Archimedes' formula for areas of spherical caps: $A=2\pi R h$, where $R$ is the radius of the sphere and $h$ is the height of a cap. For the sake of exposition, we will start with the following proposition.

\begin{proposition}\label{prop:14}
$$\tau_3\leq 8+4\sqrt{3}.$$
\end{proposition}

\begin{proof}
For a unit ball $B$ in $\mathbb{R}^3$, consider a concentric sphere with radius $\sqrt{3}$. Any unit sphere tangent to $B$ intersects this concentric sphere by a spherical cap with the angular spherical radius of $\pi/6$. The height of this spherical cap is $\sqrt{3}-3/2$. By Archimedes' formula, the area of this spherical cap is $2\pi \sqrt{3} (\sqrt{3}-3/2)=(6-3\sqrt{3})\pi$. Since the area of the concentric sphere is $12\pi$, no more than $\frac {12\pi}{(6-3\sqrt{3})\pi}=8+4\sqrt{3}$ spherical caps may fit in the surface of the concentric sphere.
\end{proof}

\begin{remark}
Of course kissing numbers are integer so any upper bound may be substituted by its integer part and Proposition \ref{prop:14} also implies that $\tau_3\leq 14$. The main purpose of the proposition is to emphasize the ideas to be transferred to the case of packings with different radii.
\end{remark}

The same idea of bounding the number of tangent spheres is not directly applicable when different radii are allowed. For a unit ball, one can construct any number of small balls tangent to it. However, for two tangent balls the smaller proportion of area taken by a smaller ball on a sphere concentric to a larger ball is compensated by a larger proportion of area taken by a larger ball on a sphere concentric to a smaller ball. This is the cornerstone of the approach by Kuperberg and Schramm.

Fix $\rho> 1$. For each closed ball $B$ denote the concentric sphere with radius $\rho$ times larger by $S_\rho(B)$. For two tangent balls $B_1$ with radius $r_1$ and $B_2$ with radius $r_2$, define $$a(B_1,B_2)=\dfrac {area(S_\rho(B_1)\cap B_2)}{area(S_\rho(B_1))}.$$

\begin{center}
\includegraphics[width=0.7\linewidth]{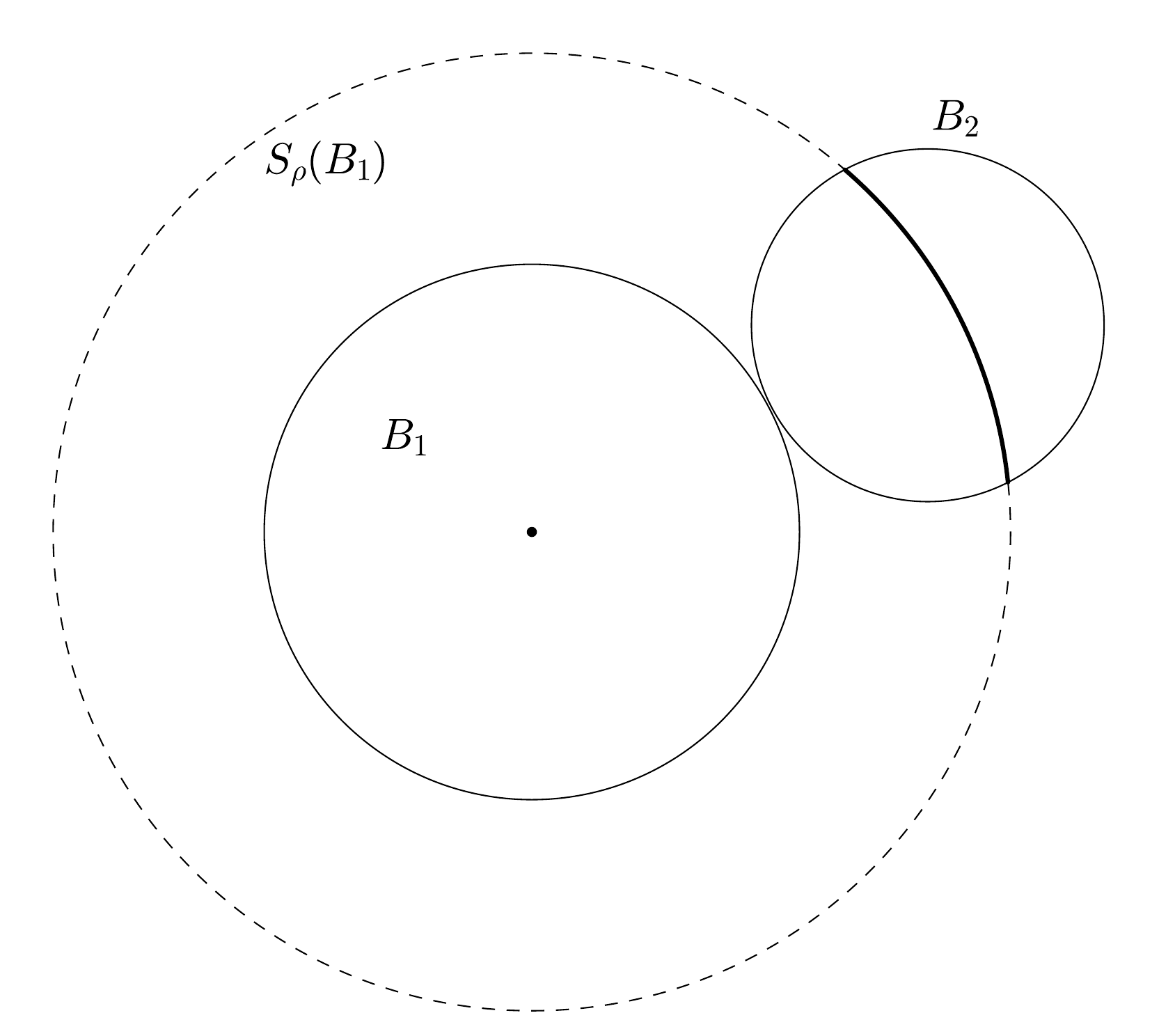}
\end{center}

Remarkably, if both $S_\rho(B_1)\cap B_2$ and $S_\rho(B_2)\cap B_1$ are non-empty, $a(B_1,B_2)+a(B_2,B_1)$ depends only on $\rho$. In order to prove this, denote the height of a spherical cap $S_\rho(B_1)\cap B_2$ by $h_1$ and the height of $S_\rho(B_2)\cap B_1$ by $h_2$ (if an intersection is empty we use 0 for its height). From this moment on, we consider only $\rho < 3$ because otherwise at least one of $a(B_1,B_2)$ and $a(B_2,B_1)$ is 0.

\begin{lemma}\label{lem:height}
$$\frac {h_1} {\rho r_1} + \frac {h_2} {\rho r_2} =\frac {-\rho^2+4\rho-3}{2\rho},$$ if both $S_\rho(B_1)\cap B_2$ and $S_\rho(B_2)\cap B_1$ are non-empty and the left hand side is greater than the right hand side otherwise.
\end{lemma}

\begin{proof}
$\frac {h_1} {\rho r_1}=1-\cos\alpha$, where $\alpha$ is the spherical radius of the cap $S_\rho(B_1)\cap B_2$.

\begin{center}
\includegraphics[width=0.7\linewidth]{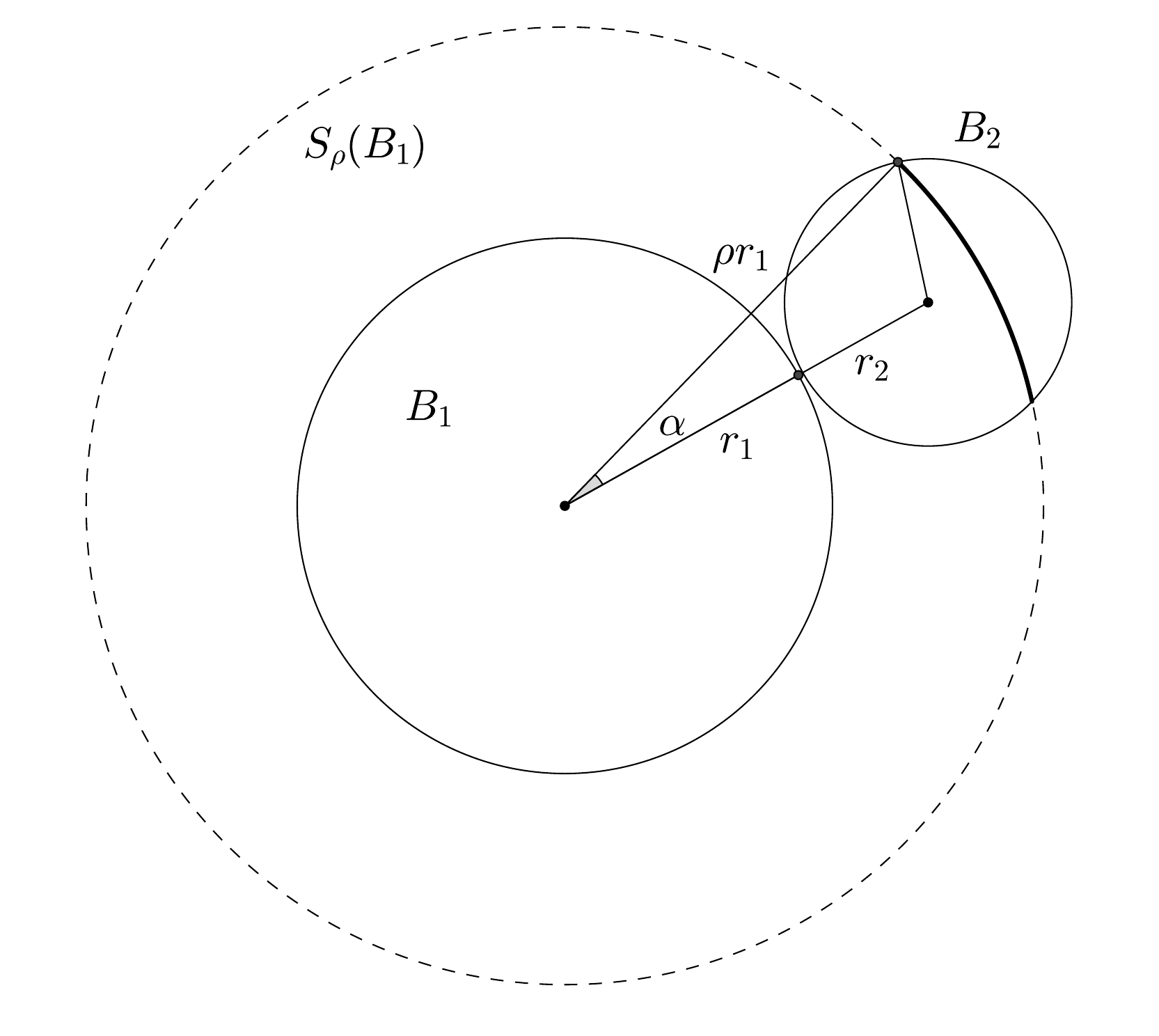}
\end{center}

We use the law of cosines for a triangle formed by the centers of $B_1$ and $B_2$ and any point on the boundary of the cap: $$(\rho r_1)^2+(r_1+r_2)^2-2 \rho r_1 (r_1+r_2) \cos\alpha=r_2^2,$$

\begin{equation}\label{formula:cap_radius}
\cos\alpha = \frac {(\rho r_1)^2+(r_1+r_2)^2- r_2^2} {2 \rho r_1 (r_1+r_2)} = \frac {(\rho^2+1) r_1 + 2r_2} {2\rho(r_1+r_2)}.
\end{equation}
Similarly, for the radius $\beta$ of the second cap we get

$$\cos\beta = \frac {(\rho^2+1) r_2 + 2r_1} {2\rho(r_1+r_2)}.$$
Therefore,

\begin{equation}\label{formula:caps}
\cos\alpha + \cos\beta= \frac {(\rho^2+1) r_1 + 2r_2} {2\rho(r_1+r_2)} + \frac {(\rho^2+1) r_2 + 2r_1} {2\rho(r_1+r_2)} = \frac {\rho^2+3}{2\rho},
\end{equation}

$$\frac {h_1} {\rho r_1} + \frac {h_2} {\rho r_2}= 2 - (\cos\alpha + \cos\beta)= \frac {-\rho^2+4\rho-3}{2\rho}.$$
Since $\rho<3$, the equality still holds when the second spherical caps consists of one point, i.e. $r_1+2r_2=\rho r_1$. Increasing $r_1$, we keep the second term equal to 0 and only increase the first one. Therefore, in the case when one intersection is empty the inequality holds.
\end{proof}

\begin{lemma}\label{lem:area3d}
$$a(B_1,B_2)+a(B_2,B_1)=\frac {-\rho^2+4\rho-3}{4\rho},$$ if both $S_\rho(B_1)\cap B_2$ and $S_\rho(B_2)\cap B_1$ are non-empty and the left hand side is greater than the right hand side otherwise.
\end{lemma}

\begin{proof}
This lemma follows immediately from the previous lemma and Archimedes' formula for areas of spherical caps.
\end{proof}

Lemmas \ref{lem:height}-\ref{lem:area3d} were essentially used by Kuperberg and Schramm, cf. equation (7) in \cite{kup94}.  

Denote by $dens(\rho)$ the supremum of $\sum_i a(B,B_i)$, where the supremum is taken over all sets $\{B_i\}$ of closed balls with disjoint interiors such all $B_i$ are tangent to $B$. If $G=(V,E)$ is a contact graph of a ball packing then, on the one hand,
$$\sum\limits_{\{X,Y\}\in E} (a(X,Y)+a(Y,X))\geq \frac {-\rho^2+4\rho-3}{4\rho} |E|.$$
On the other hand,
$$\sum\limits_{\{X,Y\}\in E} (a(X,Y)+a(Y,X)) \leq dens(\rho) |V|$$ so $$2|E|/|V|\leq \frac {8\rho} {-\rho^2 +4\rho +3}\ dens(\rho)$$
and, therefore, we have proven the following bound.

\begin{theorem}\label{thm:gen3}
$$k_3\leq \inf\limits_{1<\rho<3} \left\{\frac {8\rho} {-\rho^2 +4\rho -3}\ dens(\rho)\right\},$$
where $dens(\rho)$ denotes the supremum over proportions of area of $S_\rho(B)$ covered by non-overlapping balls tangent to $B$.
\end{theorem}

Kuperberg and Schramm used $dens(\rho)\leq 1$ and, taking the optimum $\rho=\sqrt{3}$, proved their upper bound.

\section{Bounds in higher dimensions}\label{sect:high}

We use the same notations $S_\rho (B)$, $a(B_1, B_2)$, etc as in Section \ref{sect:kup-sch}. Throughout this section we use the following formula for the $(d-1)$-dimensional area of a spherical cap with spherical radius $\alpha$ on the unit sphere in $\mathbb{R}^d$ \cite{li11}:

\begin{equation}\label{formula:cap_area}
A=\frac {\pi^{d/2}} {\Gamma(d/2)} \int\limits_0^{\sin^2\alpha} t^{\frac{d-3}2}(1-t)^{-\frac{1}{2}} dt.
\end{equation}
For $d=3$, this formula is equivalent to Archimedes' formula.

\begin{lemma}\label{lem:gen}
For a fixed $\rho$, $1<\rho<3$, $a(X,Y)+a(Y,X)$ reaches its minimum when $X$ and $Y$ are congruent.
\end{lemma}

\begin{proof}
We begin with the case when both $a(X,Y)$ and $a(Y,X)$ are not 0. Denote the radii of the spherical caps $S_\rho(X)\cap Y$ and $S_\rho(Y)\cap X$ by $\alpha$ and $\beta$, respectively. Using formula (\ref{formula:cap_area}) we get

$$a(X,Y)+a(Y,X) = K\left(\int\limits_0^{\sin^2\alpha} t^{\frac{d-3}2}(1-t)^{-\frac{1}{2}} dt + \int\limits_0^{\sin^2\beta} t^{\frac{d-3}2}(1-t)^{-\frac{1}{2}} dt\right),$$
where $K$ is a constant depending only on $d$. Radii $\alpha$ and $\beta$ must satisfy formula (\ref{formula:caps}) showed in the proof of Lemma \ref{lem:height}: $\cos\alpha + \cos\beta = C$, where $C$ is a constant depending only on $\rho$, $C\in (1,2)$. We denote $\cos\alpha$ by $x$, $x\in[C-1,1]$ so that $x\leq 1$ and $C-x\leq 1$. Then $\cos\beta=C-x$ and the value to be optimized may be rewritten as

$$g(x) := \frac 1 K (a(X,Y)+a(Y,X)) = \int\limits_0^{1-x^2} t^{\frac{d-3}2}(1-t)^{-\frac{1}{2}} dt + \int\limits_0^{1-(C-x)^2} t^{\frac{d-3}2}(1-t)^{-\frac{1}{2}} dt.$$

$$g'(x)=(-2x)(1-x^2)^{\frac {d-3}{2}}x^{-1} + 2(C-x) (1-(C-x)^2)^{\frac {d-3} {2}} (C-x)^{-1} =$$

$$= 2(-(1-x^2)^{\frac {d-3}{2}}+(1-(C-x)^2)^{\frac {d-3} {2}}).$$
Therefore, $g$ is decreasing when $1-x^2\geq 1-(C-x)^2$, i.e. when $x\in[C-1,C/2]$, and increasing when $x\in[C/2,1]$. The only minimum is attained at $x=C/2$ or, in other words, when $\alpha=\beta$ and the radii of $X$ and $Y$ are the same.

The case when one of the intersections is empty can be explained using the same argument as in Lemma \ref{lem:height}. The value of $a(X,Y)+a(Y,X)$ in the case one spherical cap consists of only one point, i.e. $\rho r_1=r_1+2r_2$, is not smaller than the observed minimum. When we increase $r_1$ one of $a(X,Y)$ increases and the other retains its 0 value. Therefore, the total value is even larger.
\end{proof}

\begin{remark}
Note that, for $d=3$, $g'(x)=0$ so $g(x)$ is constant and this lemma generalizes Lemmas \ref{lem:height} and \ref{lem:area3d}
\end{remark}

We can find the minimum established by Lemma \ref{lem:gen} explicitly. By formula (\ref{formula:caps}), $\cos\alpha=\cos\beta=\frac{\rho^2+3}{4\rho}$. The $(d-1)$-dimensional area of the unit sphere by formula (\ref{formula:cap_area}) can be found as

$$\frac {2\pi^{d/2}} {\Gamma(d/2)} \int\limits_0^{1} t^{\frac{d-3}2}(1-t)^{-\frac{1}{2}} dt.$$
Hence

$$a(X,Y)+a(Y,X)\geq \dfrac {\int\limits_0^{1-\left(\frac{\rho^2+3}{4\rho}\right)^2} t^{\frac{d-3}2}(1-t)^{-\frac{1}{2}} dt} {\int\limits_0^{1} t^{\frac{d-3}2}(1-t)^{-\frac{1}{2}} dt}.$$
Denote this minimum by $f_d(\rho)$. Then, similarly to the 3-dimensional case, we get the general bound.

\begin{theorem}\label{thm:high}
$$k_d\leq \inf\limits_{1<\rho<3} \left\{ \frac 2 {f_d(\rho)}\ dens_d(\rho)\right\},$$
where $dens_d(\rho)$ denotes the supremum over proportions of area of $S_\rho(B)$ covered by non-overlapping balls tangent to $B$.
\end{theorem}

\begin{proof}
For a contact graph $G=(V,E)$,

$$\sum\limits_{\{X,Y\}\in E} (a(X,Y)+a(Y,X))\geq f_d(\rho) |E|,$$

$$\sum\limits_{\{X,Y\}\in E} (a(X,Y)+a(Y,X)) \leq dens_d(\rho) |V|.$$
Therefore,

$$2|E|/|V|\leq \frac {2} {f_d(\rho)}\ dens_d(\rho).$$
\end{proof}

$f_d(\rho)$, as a function of $\rho$, reaches its maximum when $1-\left(\frac{\rho^2+3}{4\rho}\right)^2$ is maximal, i.e. $\rho=\sqrt{3}$ and $S_\rho(X)\cap Y$ is a spherical cap with radius $\pi/6$. Using $dens_d(\sqrt{3})\leq 1$ and Theorem \ref{thm:high} for $\rho=\sqrt{3}$, we get the bound analogous to the Kuperberg-Schramm bound in higher dimensions.

\begin{theorem}\label{thm:high_bound}
$$k_d\leq a(d)=\frac {2\int_0^1 t^{\frac {d-3}{2}} {(1-t)^{-\frac 1 2}} dt} {\int_0^{1/4} t^{\frac {d-3}{2}} {(1-t)^{-\frac 1 2}}dt}.$$
\end{theorem}

\begin{remark}\label{rem:area}
Just like the Kuperberg-Schramm upper bound is a generalization of Proposition \ref{prop:14}, this theorem is a direct generalization of the upper bound for kissing numbers $\tau_d\leq a(d)$ based on area estimates.
\end{remark}

For $d=4,5$, this theorem gives the new upper bounds on $k_d$.

\begin{corollary}
$k_4< 34.69$, $k_5< 77.76$.
\end{corollary}

\begin{proof}
We use MATLAB to calculate numerically $a(4)$ and $a(5)$: $a(4)< 34.69$, $a(5)< 77.76$.
\end{proof}

Starting from 6, upper bounds based on kissing numbers become better:
$a(6) \approx 170.58$; $a(7) \approx 368.74$; $a(8) \approx 788.65$.

As mentioned in Remark \ref{rem:area}, the bound of Theorem \ref{thm:high_bound} coincides with the bound for kissing numbers based on area estimates so it is asymptotically worse than the Kabatyanskii-Levenshtein bound from \cite{kab78}.

\section{New bound in dimension 3}\label{sect:3d}

The area argument is arguably the easiest way to get upper bounds on kissing numbers. Sections \ref{sect:kup-sch} and \ref{sect:high} essentially explain how to extend this argument to the situation of packings with different radii. It is reasonable to try extending more sophisticated methods of analyzing kissing numbers to the more general case of different radii.

One of the fruitful approaches in this direction goes back to Fejes~T{\'o}th (see \cite{fej43}). The idea consists of constructing a certain tiling associated with a packing (typically, a Delaunay-like or Voronoi-like tiling) and bounding the density of the packing in each tile of a tiling. This bound is then used as a general bound on the density.

In his original paper, Fejes~T{\'o}th showed that the density of a packing of congruent circles of spherical radius $\alpha$ in the unit sphere is not greater than the density of this packing in the regular spherical triangle of side length $2\alpha$ with centers of circles at the vertices of this triangle. Coxeter in \cite{cox63} conjectured that an analogous bound (sometimes also known as the simplex bound) will be true in higher dimensions as well and found an explicit expression for it. Finally, B{\"o}r{\"o}czky in \cite{bor78} proved this bound for all spaces of constant curvature using subdivisions into quasi-orthoschemes (refinements of Delaunay and Voronoi tilings).

We use the theorem that immediately follows from the results of Florian in \cite{flo01, flo07} generalizing \cite{fej43} for the case of spherical caps of different sizes. 

\begin{theorem}\label{thm:florian}
Let $K(\alpha)$ be a non-decreasing function defined on $I=[\alpha_{min},\alpha_{max}]$, $0<\alpha_{min}\leq \alpha_{max}\leq \frac \pi 2$. For a packing $\mathcal{C}$ of a unit sphere with circles whose radii belong to $I$, the density is defined as 

$$d(\mathcal{C})=\frac 1 {4\pi} \sum\limits_{C\in\mathcal{C}} K(radius(C)).$$
For $x,y,z\in I$, we consider a spherical triangle $\Delta$ formed by centers of pairwise tangent circles of radii $x,y,z$. The density of this triangle is defined by

$$D(x,y,z)=\frac 1 {2\pi\cdot area(\Delta)} \left(K(x)\angle x + K(y)\angle y+K(z)\angle z\right),$$
where $\angle x$, $\angle y$, and $\angle z$ are measures of spherical angles with vertices at centers of circles of radii $x,y$, and $z$, respectively.

Then $d(\mathcal{C})\leq \max\limits_{x,y,z\in I} D(x,y,z)$.
\end{theorem}

We can think of $K$ as a weight function so $\sum K(radius(C))$ is the total weight of all spherical caps in a packing. Then $d(\mathcal{C})$ represents the total weight density of the packing. On the other hand, for a spherical angle $\angle x$, the part of the spherical cap of radius $x$ that is inside the angle is $\frac {\angle x} {2\pi}$ so its weight is $K(x) \frac {\angle x} {2\pi}$. The total weight of the triangle may be calculated as
$$\frac 1 {2\pi} \left(K(x)\angle x + K(y)\angle y+K(z)\angle z\right)$$
and the triangle's weight density is precisely
$$\frac 1 {2\pi\cdot area(\Delta)} \left(K(x)\angle x + K(y)\angle y+K(z)\angle z\right).$$
The main conclusion of the theorem is that, in order to bound the maximum weight density, it is sufficient to consider only triangles formed by three pairwise tangent spherical caps.

The proof of this theorem essentially consists of two parts. First, we can show that for any saturated packing with caps of radii between $\alpha_{min}$ and $\alpha_{max}$, there exists a Delaunay-like (Moln{\'a}r) triangulation (see \cite{mol67}). The second part consists of proving that the maximal density among Delaunay-like triangles is attained on a triangle defined by three pairwise tangent caps.

\begin{remark}
Formally, Florian proved the theorem only for the case when $I$ is a finite set of possible radii but Theorem \ref{thm:florian} immediately follows from his results.
\end{remark}

We will couple this bound on the density with Theorem \ref{thm:gen3} to get the new bound in dimension 3. Just to recall the notation used in the previous sections, by $dens(\rho)$ we mean the supremum over proportions of area of $S_\rho(B)$ covered by non-overlapping balls tangent to $B$.

If we forget that spherical caps on $S_\rho(B)$ were initially formed by non-overlapping balls and just try to find upper bounds for an arbitrary packing by spherical caps, it is impossible to separate $dens(\rho)$ from 1. A spherical cap may have an arbitrarily small radius and thus the density of a packing may be arbitrarily close to 1.

\begin{center}
\includegraphics[scale=0.6]{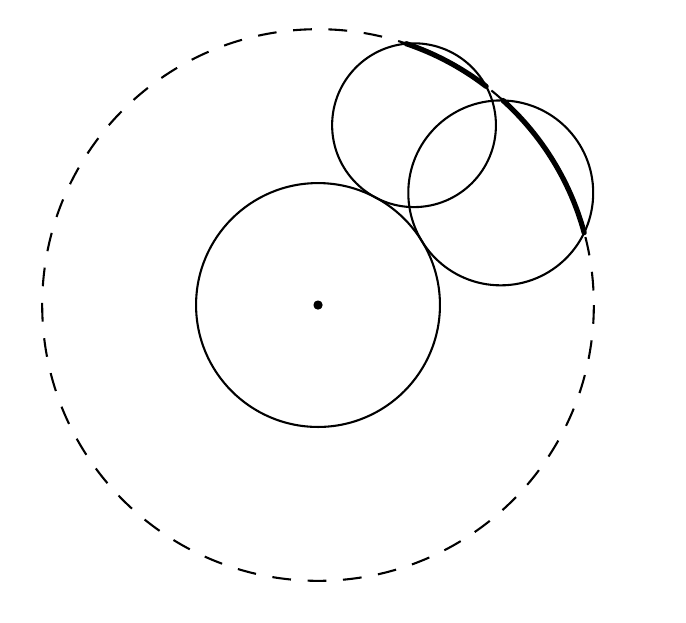}

Relatively small caps may not be too close to each other.
\end{center}

The key idea under finding a meaningful upper bound for $dens(\rho)$ is to use an auxiliary circular packing which extends the original one and, most importantly, may not contain spherical caps of arbitrarily small size. For each ball $X$ tangent to a ball $B$, we define a spherical cap $C_\rho(B,X)$ as a cap on $S_\rho(B)$ defined by common tangent planes of $B$ and $X$ if a point of tangency of such common tangent plane with $X$ is inside $S_\rho(B)$. Otherwise, $C_\rho(B,X)=S_\rho(B)\cap X$.

\begin{center}
\includegraphics[scale=0.45]{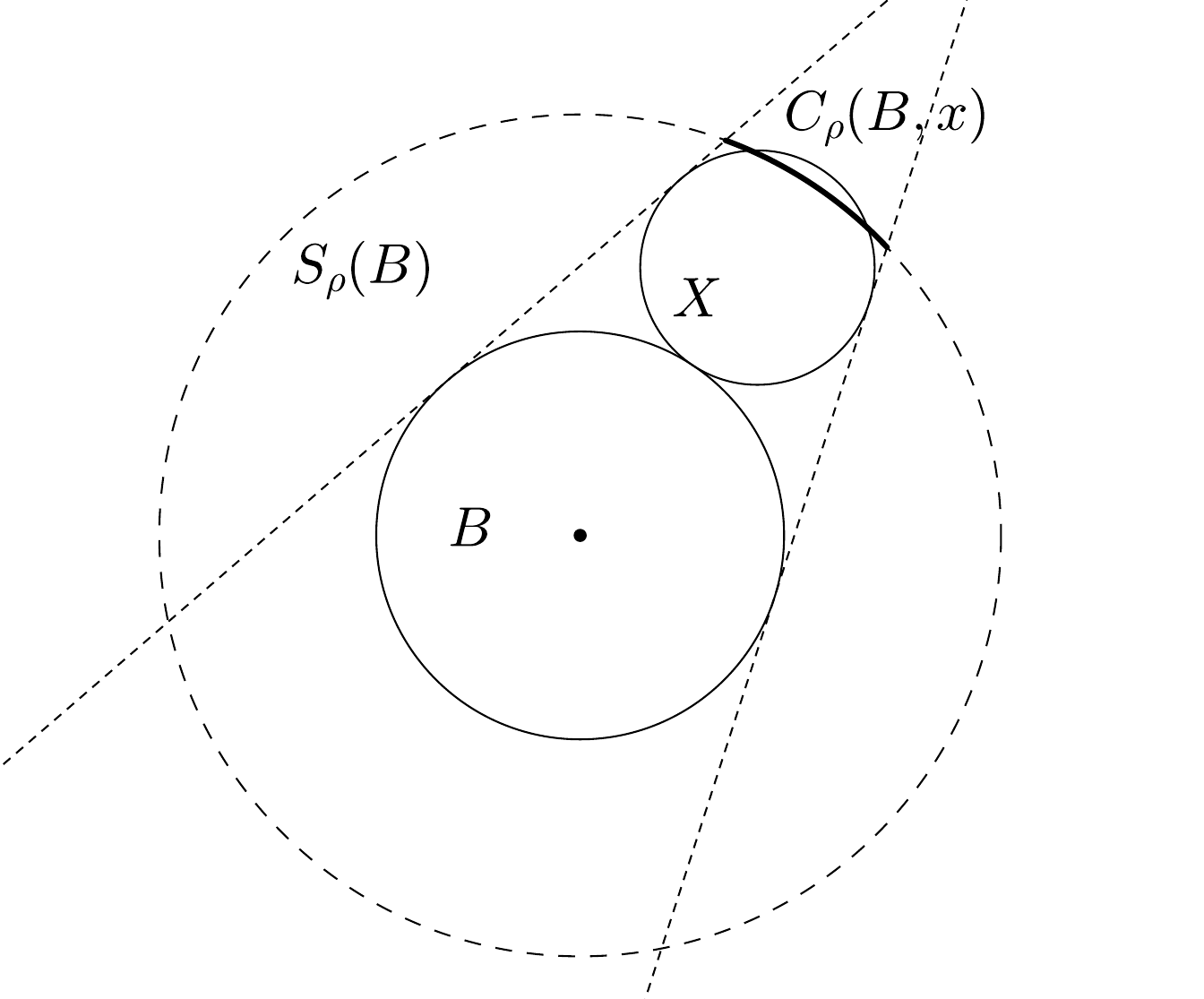}
\end{center}

\begin{lemma}\label{lem:aux}
For any $\rho>1$ and any non-overlapping balls $X$ and $Y$ tangent to $B$, spherical caps $C_\rho(B,X)$ and $C_\rho(B,Y)$ do not overlap.
\end{lemma}

\begin{center}
\includegraphics[scale=0.45]{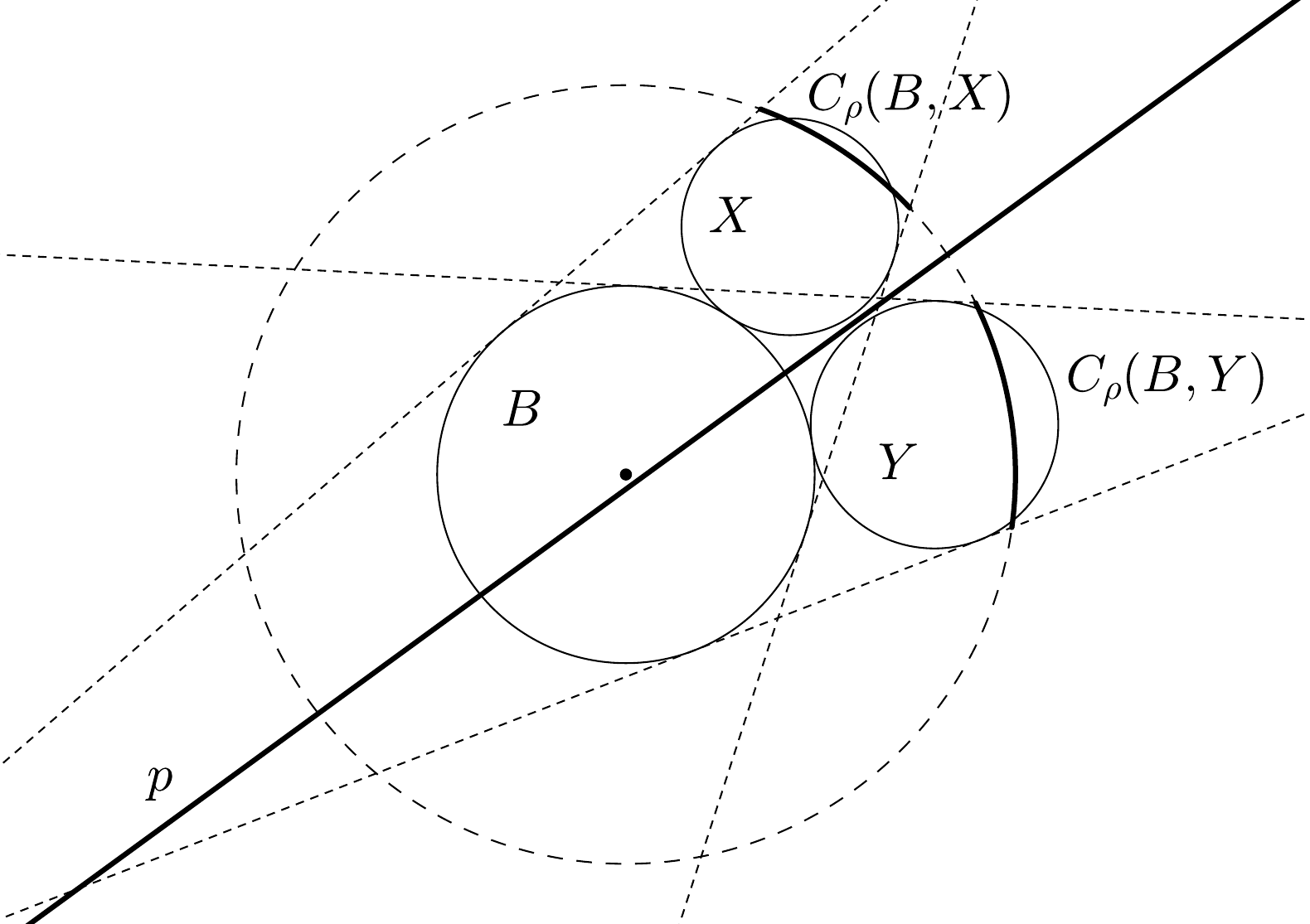}
\end{center}

\begin{proof}
Consider an arbitrary plane $p$ separating closed balls $X$ and $Y$. We want to show that $p$ separates $C_\rho(B,X)$ and $C_\rho(B,Y)$ as well. Firstly, we note that $p$ must intersect $B$. If this is not the case, one of the two open half-spaces formed by $p$ does not have a common point with $B$ but it contains either $X$ or $Y$, both of which are tangent to $B$.

$p$ may not contain interior points of $X$ or $Y$. Therefore, in order to complete the proof of the lemma, it is sufficient for us to show that $p$ does not have any interior points of $C_\rho(B,X)$ or $C_\rho(B,Y)$. We assume that $p$ has an interior point of $C_\rho(B,X)$. We connect this point with an arbitrary point of $p\cap B$ by a line segment. This segment intersects $X$ by an interior point so we get a contradiction to the fact that $p$ contains no interior points of $X$.
\end{proof}

Now we can describe the general approach of using Theorem \ref{thm:florian} in conjunction with Lemma \ref{lem:aux} to find the density bounds for packings formed by non-overlapping balls. Let $X$ be a ball tangent to the ball $B$. We use the same notations as above, i.e. $S_{\rho}(B)\cap X$ is the spherical cap formed by the intersection of $X$ with the enlarged copy of $B$ and $C_\rho(B,X)$ is its auxiliary cap. We denote the spherical radius of the auxiliary cap by $x$ and the area of the initial cap, $S_{\rho}(B)\cap X$, by $K(x)$. Here $K$ is a properly defined function on the set of all possible values of $x$ since the radii of the initial and auxiliary caps are in one-to-one correspondence with each other.

Under this setup, if there is a packing $\mathcal{C}$ of auxiliary caps then $d(\mathcal C)$, defined as $\frac 1 {4\pi} \sum\limits_{C\in\mathcal{C}} K(radius(C))$, is precisely the density of the packing formed by initial caps. We can apply Theorem \ref{thm:florian} to find upper bounds on $d(\mathcal C)$. In order to do this, we need explicit expressions of $\alpha_{min}$, $\alpha_{max}$, and $K$.

\begin{center}
\includegraphics[scale=0.45]{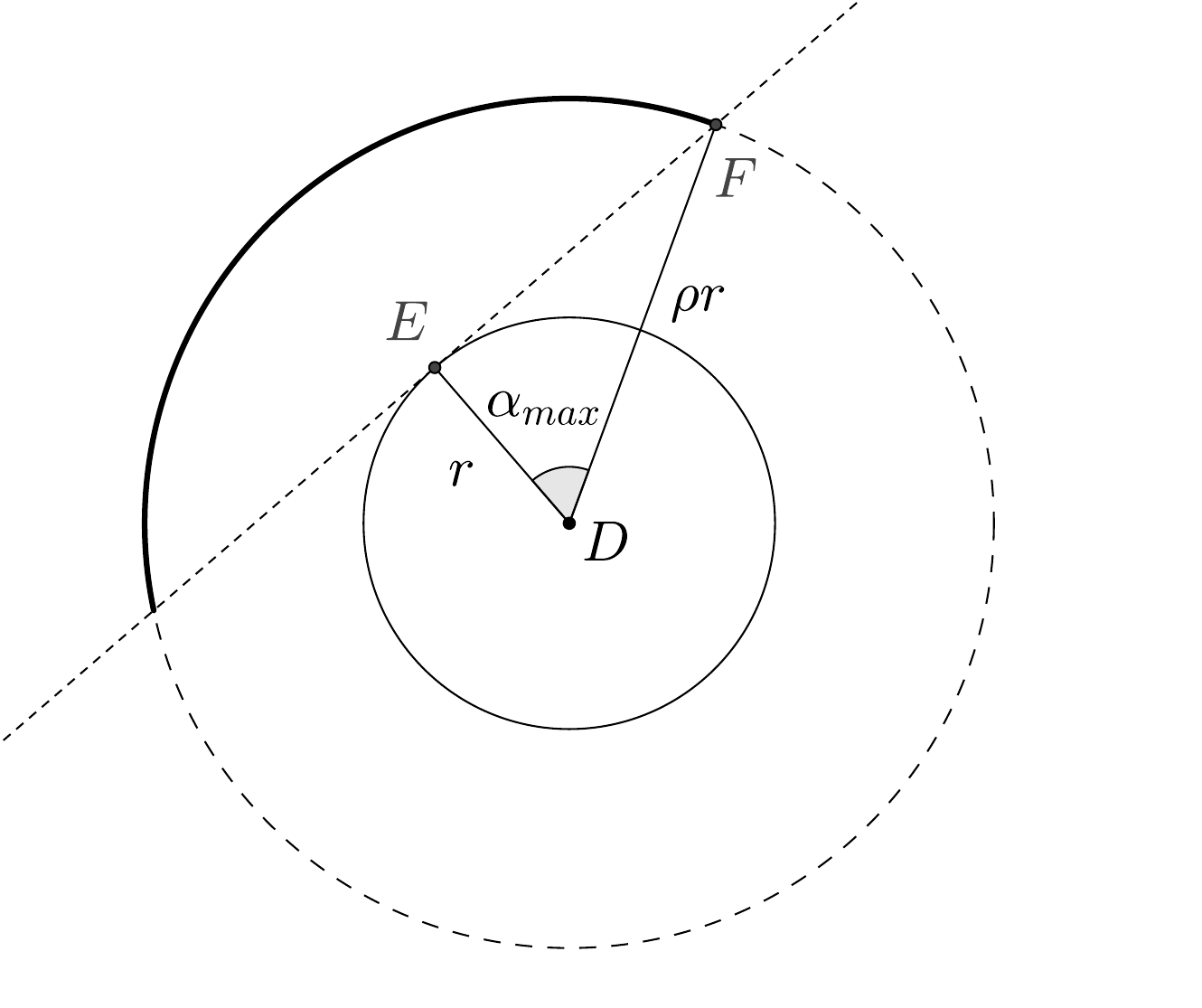}
\end{center}

\begin{equation}\label{formula:alpha_max}
\alpha_{max} = \angle EDF = \arccos{\frac 1 \rho}.
\end{equation}

Here for $\alpha_{max}$ we should take the maximal radius of an auxiliary cap. Formally, in our setup there is no maximal value of the cap radius. The ball $X$ tangent to $B$ may be of any positive radius. We extend the problem by compactifying the set of possible spherical caps and allowing the ball $X$ to be infinite. This means that the maximal spherical cap is formed by the intersection of the closed half-space $X$ tangent to $B$ with the enlarged concentric sphere $S_{\rho}(B)$.

\begin{center}
\includegraphics[scale=0.45]{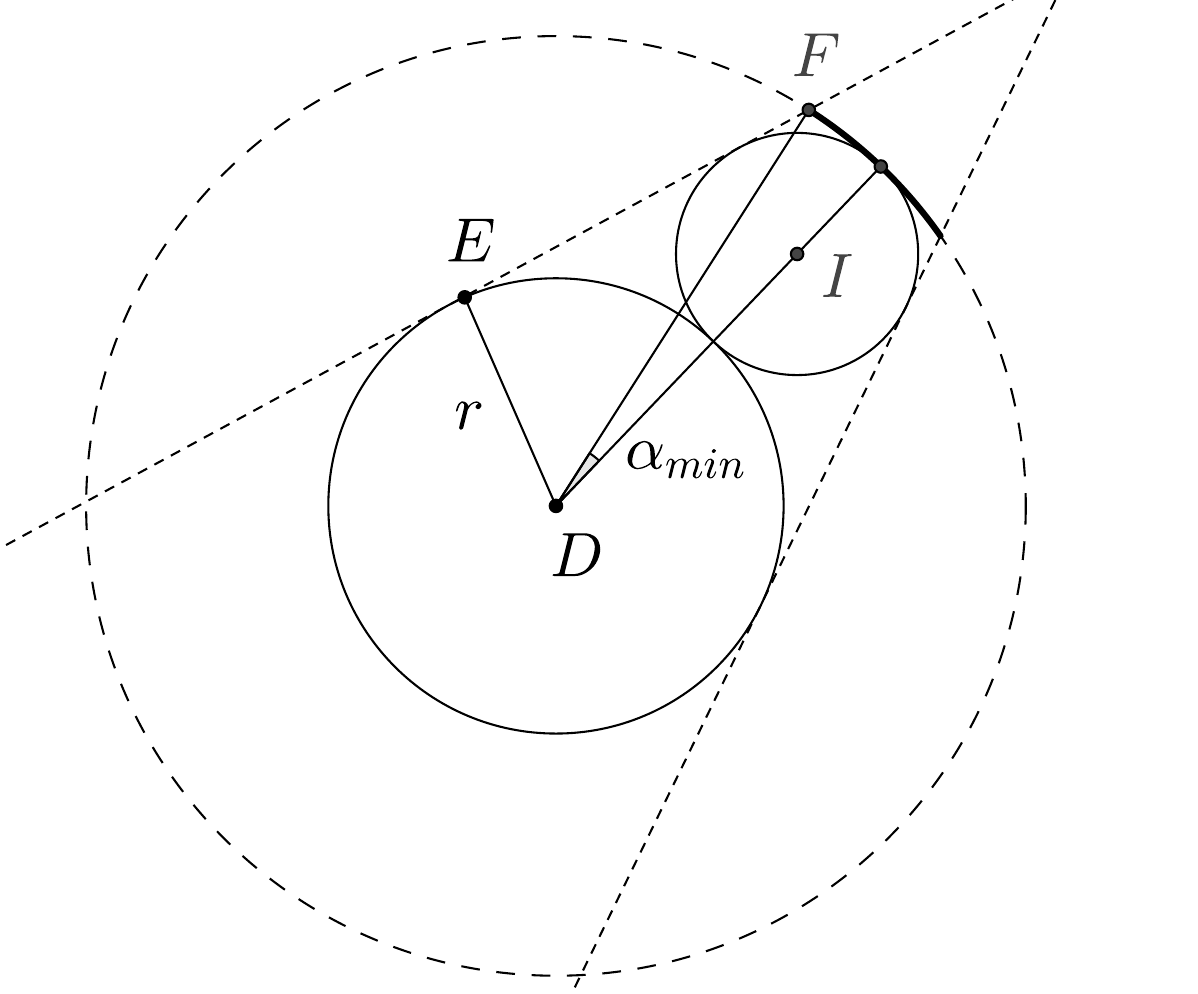}
\end{center}

\begin{equation}\label{formula:alpha_min}
\alpha_{min}=\angle EDI - \angle EDF = \arccos \frac {3-\rho} {1+\rho} - \arccos{\frac 1 \rho}.
\end{equation}

The value of $\alpha_{min}$, the minimal radius of an auxiliary cap, is defined by the ball $X$ whose intersection with $S_{\rho}(B)$ is precisely one point. All smaller balls will have no contribution in $d(\mathcal C)$.

\begin{center}
\includegraphics[scale=0.45]{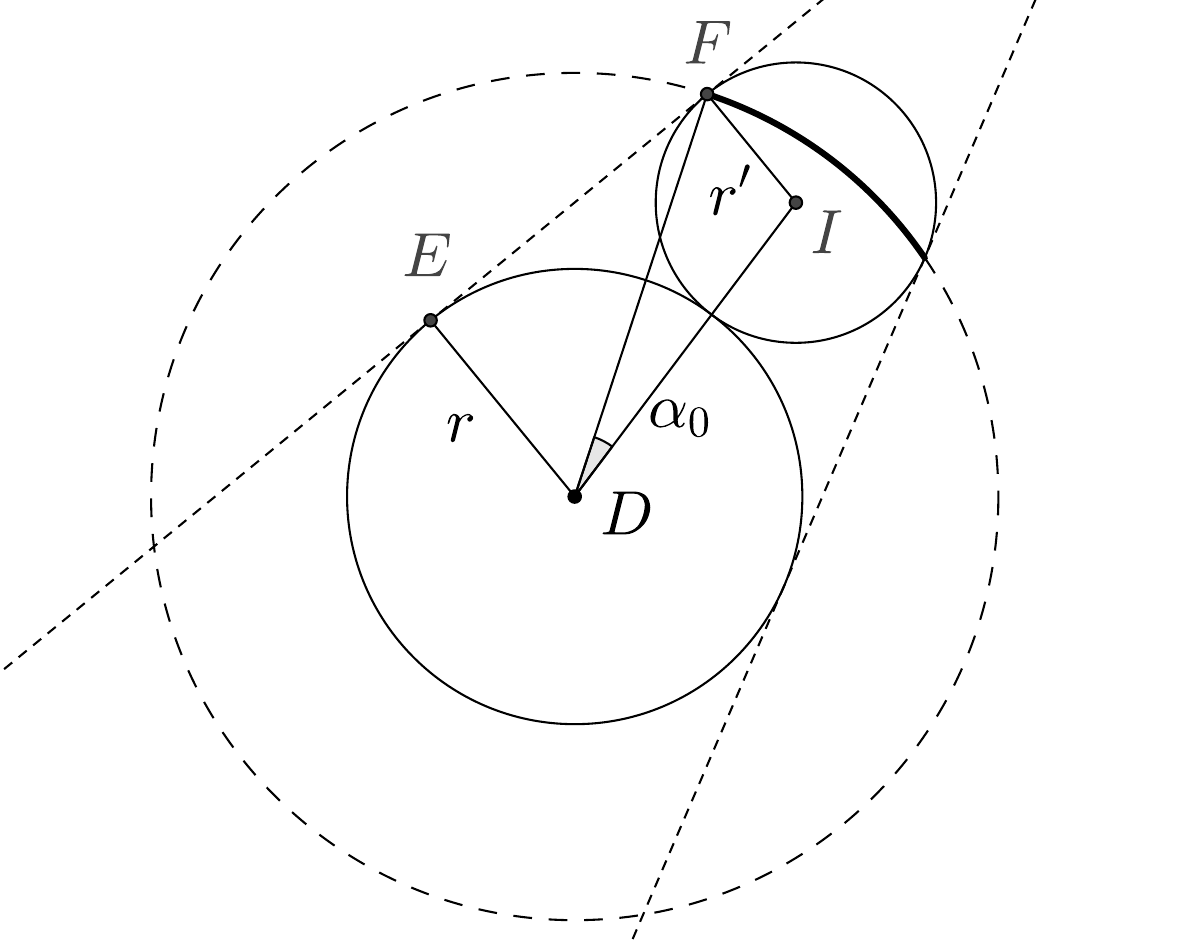}
\end{center}

By definition, the radii of initial and auxiliary caps coincide if the common tangent plane of the ball $B $ and a ball $X$ tangent to $B$ lies inside $S_\rho(B)$. This will happen if $\alpha$ is greater or equal to a certain threshold $\alpha_0$, when the tangent point is exactly on $S_\rho(B)$.

From $\triangle DEF$ we can find that $r'=\frac{\rho^2-1}{4} r$. Then, using formula (\ref{formula:cap_radius}), we get the formula for $\alpha_0$:

\begin{equation}\label{formula:alpha_0}
\alpha_0=\arccos \frac {3\rho^2+1}{\rho(\rho^2+3)}.
\end{equation}
Hence we can define $K(\alpha)$ as the area of a spherical cap with the spherical radius $\alpha$ for $\alpha\in[\alpha_0,\alpha_{max}]$:

\begin{equation}\label{formula:K1}
K(\alpha)=2\pi (1-\cos\alpha) \text{ if } \alpha\in[\alpha_0,\alpha_{max}].
\end{equation}

\begin{center}
\includegraphics[scale=0.45]{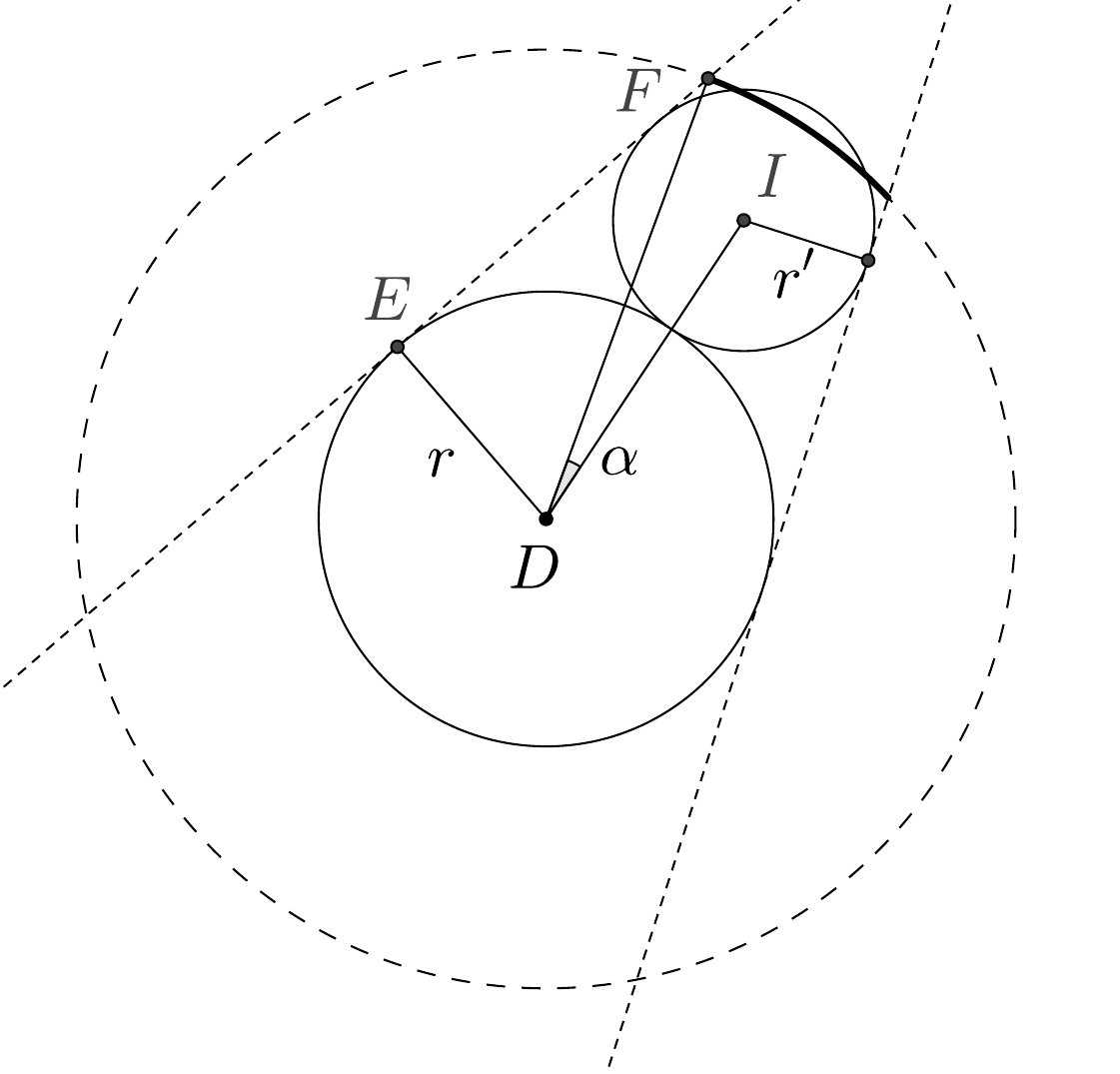}
\end{center}
For the case when the initial cap and its auxiliary cap do not coincide, on the one hand, $\cos\angle EDI = \cos (\angle EDF + \alpha) = \frac 1 \rho \cos \alpha - \sqrt{1 - \frac 1 {\rho^2}}\sin \alpha$. On the other hand, $\cos\angle BAI = \frac {r-r'} {r+r'}$. Hence $r'= \left(\dfrac 2 {\frac 1 \rho \cos \alpha - \sqrt{1 - \frac 1 {\rho^2}}\sin \alpha+1} -1 \right) r.$ Combining this with formula (\ref{formula:cap_radius}), we find $K(\alpha)$ for $\alpha\in[\alpha_{min},\alpha_0]$;

\begin{equation}\label{formula:K2}
K(\alpha)=2\pi \left(1-\dfrac {(\rho^2-1)\left(\frac 1 \rho \cos \alpha - \sqrt{1 - \frac 1 {\rho^2}}\sin \alpha+1\right) +4} {4\rho}\right) \text{ if } \alpha\in[\alpha_{min},\alpha_0].
\end{equation}
The angles $\angle x, \angle y, \angle z$ are found by the Spherical Law of Cosines:

\begin{equation}\label{formula:angle_x}
\angle x = \arccos \frac {\cos(y+z) - \cos(x+z) \cos(x+y)} {\sin(x+z) \sin(x+y)};
\end{equation}

\begin{equation}\label{formula:angle_y}
\angle y = \arccos \frac {\cos(x+z) - \cos(x+y) \cos(y+z)} {\sin(x+y) \sin(y+z)};
\end{equation}

\begin{equation}\label{formula:angle_z}
\angle z = \arccos \frac {\cos(x+y) - \cos(x+z) \cos(y+z)} {\sin(x+z) \sin(y+z)}.
\end{equation}

Finally, we can formulate the general bound in dimension 3.

\begin{theorem}\label{thm:3main}
For any $\rho\in (1,3)$, we define $D_\rho(x,y,z)$ for all triples $x,y,z\in I_\rho = [\alpha_{min},\alpha_{max}]$,

$$D_\rho(x,y,z) = \frac 1 {2\pi (\angle x +\angle y + \angle z - \pi)} \left(K(x)\angle x + K(y)\angle y+K(z)\angle z\right),$$
where $\alpha_{min}, \alpha_{max}, K(\alpha), \angle x, \angle y, \angle z$ are defined by formulas (\ref{formula:alpha_max}-\ref{formula:angle_z}). Then

$$k_3 \leq \inf\limits_{1<\rho<3}\left\{ \max\limits_{x,y,z\in I_\rho} D_\rho(x,y,z)\ \frac {8\rho} {-\rho^2 +4\rho -3}\right\}$$
\end{theorem}

\begin{proof}

Using Theorems \ref{thm:gen3} and \ref{thm:florian} we get

$$k_3\leq \inf\limits_{1<\rho<3}\left\{ dens(\rho)\ \frac {8\rho} {-\rho^2 +4\rho -3}\right\}\leq$$

$$\leq \inf\limits_{1<\rho<3}\left\{ \max\limits_{x,y,z\in I_\rho} D_\rho(x,y,z)\ \frac {8\rho} {-\rho^2 +4\rho -3}\right\}.$$

\end{proof}

We approximate the infimum in Theorem \ref{thm:3main} numerically using MATLAB. The value that is slightly less than $13.91$ %$13.908778\ldots$
is attained at $\rho=1.755$. Here we explain in more detail how this value was calculated.

In order to find the optimal $\rho$ we exclude values where the function $D_\rho(\alpha_0,\alpha_0,\alpha_0)\ \frac {8\rho} {-\rho^2 +4\rho -3}$ is at least $14$. These are found numerically using the $fzero$ function in MATLAB. The interval for suitable $\rho$'s is subsequently narrowed down to $[1.562,1.928]$. We go over $\rho$'s from this interval with the step $0.001$. For each $\rho$, we find the maximal density numerically via the $fminsearch$ function in MATLAB. The values of starting points of $fminsearch$ are taken from the grid with $0.01$ step for each coordinate of $(x,y,z)\in [\alpha_{min},\alpha_{max}]^3$. These calculations show that the minimizing $\rho$ is 1.755 and the minimal value is slightly less that 13.91. %approximately $13.908778$.

Now we can prove the new upper bound for $k_3$ with computer assistance.

\begin{corollary}
$$k_3 < 13.92.$$
\end{corollary}

\begin{proof}
We use $\rho=1.755$ and estimate $D_\rho(x,y,z)$ from above. Note that, over any compact region in $\mathbb{R}^3$, 

\begin{equation}\label{formula:estimate_D}
\max D_\rho(x,y,z) \leq  \frac 1 {2\pi \min area(x,y,z)} \left(\max K(x) \max \angle x +\max K(y) \max \angle y +\max K(z) \max \angle z  \right).
\end{equation}

We subdivide $I_\rho^3$ into cubes $[a,a+\delta]\times [b,b+\delta]\times [c,c+\delta]$. Straightforward calculations show that, over such a cube, the minimum area is attained at $(a,b,c)$, the maximum $K(x)$ is attained at $a+\delta$ (similarly, for the maxima of $K(y)$ and $K(z)$), the maximum $\angle x$ is attained at $(a,b+\delta,c+\delta)$ if $2x+y+z\leq\pi-4\delta$, at $(a+\delta,b+\delta,c+\delta)$ if $2x+y+z\geq\pi$, and at one of these two points for rare cases when $\pi-4\delta<2x+y+z<\pi$ (similarly, for the maxima of $\angle y$ and $\angle z$). Using $\delta=0.00005$ and checking values from inequality (\ref{formula:estimate_D}) for all cubes of the subdivision via computer, we get that $D_\rho(x,y,z)\ \frac {8\rho} {-\rho^2 +4\rho -3} < 13.92$.
\end{proof}

%\begin{remark}
%Using better estimates than in inequality (\ref{formula:estimate_D}) or smaller values for $\delta$ one should get a bound closer to the actual value $13.908778\ldots$ obtained numerically.
%\end{remark}

\section{Discussion}\label{sect:discuss}

In this section we would like to list several general observations and directions for research in this area.

\begin{enumerate}

\item First of all, we note that the approach utilized in the paper actually solves a more general problem. In a packing, any two elements do not overlap. Instead of this condition we can require a weaker condition: for any ball in a family $\mathcal{F}$, all balls from $\mathcal{F}$ tangent to it do not overlap. All upper bounds for the average degree of contact graphs will be valid for such families as well. It will be interesting to find an argument taking into account the actual packing condition.

\item As we can see, some methods used for finding upper bounds can be transferred to bound the average degree of packings with different radii. Arguably, the most successful of these methods is based on zonal spherical functions and linear or semidefinite programming (see \cite{kab78, mus08, bac08, mit10}). It seems feasible to use some sort of averaging argument and extend the bounds obtained by Delsarte's method (see \cite{del73, del77}) to the case of different radii.

\begin{remark}
After the paper was accepted for publication, Dostert, Kolpakov, and Oliveira \cite{dos20} used the general setup of Sections \ref{sect:kup-sch}-\ref{sect:high} and the semidefinite programming approach to find new upper bounds for $k_d$ in dimensions $3,\ldots, 9$.
\end{remark}

\item Unfortunately, there are no higher-dimensional analogues of Florian's results. The proof in \cite{bor78} is quite heavy technically and cannot be directly extended to the case of different radii.

\item Since the kissing case for congruent balls is essentially a particular case of the average kissing number problem, any area-based approach cannot bring an upper bound better than $\approx 13.397$ (Fejes T\'{o}th -- Coxeter -- B\"{o}r\"{o}czky simplex bound).

\end{enumerate}

\section{Acknowledgments}

The author would like to thank Arseniy Akopyan who brought this problem to his attention and the anonymous referee whose comments have improved the paper immensely. The author was supported in part by NSF grant DMS-1400876.

\bibliographystyle{amsalpha}

\begin{thebibliography}{A}

\bibitem{alo97}
{\sc N. Alon.}
Packings with large minimum kissing numbers.
\emph{Discrete Mathematics} 175.1 (1997), 249--251.

\bibitem{and70a}
{\sc E. M. Andreev.}
On convex polyhedra in Lobacevskii space.
\emph{Math. USSR Sbornik} 10, 3 (1970), 413--440.

\bibitem{and70b}
{\sc E. M. Andreev.}
On convex polyhedra of finite volume in Lobacevskii space.
\emph{Math. USSR Sbornik} 12, 2 (1970), 270--259.

\bibitem{bac08}
{\sc C. Bachoc, F. Vallentin.}
New upper bounds for kissing numbers from semidefinite programming.
\emph{J. Amer. Math. Soc.} 21 (2008), 909--924.

\bibitem{ben13}
{\sc I. Benjamini, O. Schramm.}
Lack of sphere packing of graphs via nonlinear potential theory.
\emph{Journal of Topology and Analysis} 5.01 (2013), 1--11.

\bibitem{bez18}
{\sc K. Bezdek, M. A. Khan.}
Contact numbers for sphere packings.
\emph{New Trends in Intuitive Geometry}, in Bolyai Society Mathematical Studies 27, Springer (2018), 25--47.

\bibitem{bor78}
{\sc K. B{\"o}r{\"o}czky.}
Packing of spheres in spaces of constant curvature.
\emph{Acta Math. Acad. Sci. Hungar.}, 32 (1978), 243--261.

\bibitem{bre96}
{\sc H. Breu, D. Kirkpatrick.}
On the complexity of recognizing intersection and touching graphs of disks.
\emph{Graph Drawing}, Lecture Notes in Computer Science, 1027, Springer, Berlin (1996), 88--98.

\bibitem{che17}
{\sc H. Chen.}
Ball packings with high chromatic numbers from strongly regular graphs. 
\emph{Discrete Mathematics} 340, 7 (2017), 1645--1648.

\bibitem{cox63}
{\sc H. S. M. Coxeter.}
An upper bound for the number of equal nonoverlapping spheres that can touch another of the same size.
In \emph{Proc. {S}ympos. {P}ure {M}ath.} {V}ol. {VII}, AMS (1963), 53--71.

\bibitem{del73}
{\sc P.~Delsarte.}
An algebraic approach to the association schemes of coding theory.
\emph{Philips Res. Rep. Suppl.}, (10):vi+97, 1973.

\bibitem{del77}
{\sc P.~Delsarte, J.~M. Goethals, J.~J. Seidel.}
Spherical codes and designs.
\emph{Geometriae Dedicata}, 6(1977), 363--388.

\bibitem{dos20}
{\sc  M. Dostert, A. Kolpakov, F.M. de Oliveira Filho.}
Semidefinite programming bounds for the average kissing number.
arXiv:2003.11832 (2020)

\bibitem{epp03}
{\sc D. Eppstein, G. Kuperberg, G. Ziegler.}
Fat 4-polytopes and fatter 3-spheres.
\emph{Discrete Geometry: In honor of W. Kuperberg's 60th birthday, Pure and Appl. Math.}
253 (2003), 239--265.

\bibitem{fej43}
{\sc L. Fejes~T{\'o}th.}
{\"U}ber eine {A}bschatzung des k{\"u}rzesten {A}bstandes zweier {P}unkt eines auf einer {K}ugelfl{\"u}che liegenden {P}unksystems.
\emph{Jber. Deutschen Math. Verein} 53 (1943), 65--68.

\bibitem{flo01}
{\sc A. Florian.}
Packing of Incongruent Circles on the Sphere.
\emph{Monatsh. Math.} 133 (2001), 111--129.

\bibitem{flo07}
{\sc A. Florian.}
Remarks on my paper: packing of incongruent circles on the sphere
\emph{Monatsh. Math.} 152 (2007), 39--43.

\bibitem{hli97}
{\sc P. Hlin{\v{e}}n{\'y}.}
Touching graphs of unit balls.
\emph{Graph Drawing}, Lecture Notes in Computer Science, 1353, Springer, Berlin (1997), 350--358.

\bibitem{hli01}
{\sc P. Hlin{\v{e}}n{\'y}, J. Kratochv{\i}l}
Representing graphs by disks and balls (a survey of recognition-complexity results).
\emph{Discrete Mathematics} 229 (2001), 101--124.

\bibitem{kab78}
{\sc G. A. Kabatyanskii, V. I. Levenshtein.}
Bounds for packings on a sphere and in space.
\emph{Probl. Inform. Transm.} 14 (1978), 1--17.

\bibitem{koe36}
{\sc P. Koebe.}
Kontaktprobleme der konformen Abbildung.
\emph{Ber. Verh. S\"{a}chs. Akad. Leipzig} 88 (1936), 141--164.

\bibitem{kup94}
{\sc G. Kuperberg, O. Schramm.}
Average kissing numbers for non-congruent sphere packings
\emph{Math. Res. Lett.} 1 (1994), 339--344.

\bibitem{lee64}
{\sc J. Leech.}
Some sphere packings in higher space.
\emph{Can. J. Math} 16 (1964), 657--682.

\bibitem{lev79}
{\sc V. I. Levenshtein.}
On bounds for packing in n-dimensional Euclidean space.
\emph{Soviet Math. Dokl.} 20 (1979), 417--421.

\bibitem{li11}
{\sc S. Li.}
Concise formulas for the area and volume of a hyperspherical cap.
\emph{Asian Journal of Mathematics and Statistics}, 4 (2011), 66--70.

\bibitem{mae07}
{\sc H. Maehara.}
On configurations of solid balls in 3-space: chromatic numbers and knotted cycles.
\emph{Graphs Combin.} 23 (2007), no. suppl. 1, 307--320.

\bibitem{mil97}
{\sc G. L. Miller, S.-H. Teng, W. P. Thurston, S. A. Vavasis.}
Separators for sphere-packings and nearest neighbor graphs.
\emph{J. ACM} 44 (1997), no. 1, 1--29.

\bibitem{mit10}
{\sc H. D. Mittelman, F. Vallentin.}
High accuracy semidefinite programming bounds for kissing numbers.
\emph{Experimental Math.} 19, 2010, 174--178.

\bibitem{mol67}
{\sc J. Moln{\'a}r.}
Kreispackungen und Kreis{\"u}berdeckungen auf Fl{\"a}chen konstanter Kr{\"u}mmung.
{\emph Acta Mathematica Hungarica} 18 (1967), 243--251.

\bibitem{mus08}
{\sc O. R. Musin.}
The kissing number in four dimensions.
\emph{Ann. of Math.} 168 (2008), 1--32.

\bibitem{odl79}
{\sc A. M. Odlyzko, N. J. A. Sloane.}
New bounds on the number of unit spheres that can touch a unit sphere in n dimensions.
\emph{J. Combin. Theory Ser. A} 26 (1979), 210--214.

\bibitem{sch53}
{\sc K. Sch\"{u}tte, B. L. van der Waerden.}
Das Problem der dreizehn Kugeln.
\emph{Math. Ann.} 125 (1953) 325--334.

\bibitem{thu88}
{\sc W. P. Thurston.}
The geometry and topology of 3-manifolds.
\emph{Princeton University Notes} (1988).

\bibitem{vla19}
{\sc S. Vl\u adu\c t.}
Lattices with exponentially large kissing numbers.
\emph{Moscow Journal of Combinatorics and Number Theory} 8 (2019), 163--177.
\end{thebibliography}

\end{document}